\documentclass[a4paper,12pt]{article}
\usepackage[utf8]{inputenc}

\usepackage{mathrsfs}
\usepackage{graphicx}
\usepackage{enumerate}
\usepackage{multicol}
\usepackage{color}
\usepackage{amsmath,amssymb,amscd}
\usepackage{amsthm}
\usepackage{mathabx}

\usepackage{pdfpages}
\usepackage{hyperref}

\usepackage{times}
\usepackage[varg]{txfonts}

\usepackage[top=1in, bottom=1in, left=0.5in, right=0.5in]{geometry}

\theoremstyle{plain}
\newtheorem{thm}{Theorem}
\newtheorem{lem}{Lemma}
\newtheorem{cor}{Corollary}
\newtheorem{prop}{Proposition}

\theoremstyle{definition}
\newtheorem{dfn}{Definition}
\newtheorem{ex}{Example}

\theoremstyle{remark}
\newtheorem{rem}{Remark}

\newtheorem*{ackn}{Acknowledgment}

\title{Monomial functions, normal polynomials and polynomial equations}
\author{Eszter Gselmann and Mehak Iqbal}

\begin{document}
\maketitle

\begin{abstract}
In this paper we consider generalized monomial functions $f, g\colon \mathbb{F}\to \mathbb{C}$ (of possibly different degree) that also fulfill 
\[
f(P(x))= Q(g(x)) 
\qquad 
\left(x\in \mathbb{F}\right), 
\]
where $P\in \mathbb{F}[x]$ and $Q\in \mathbb{C}[x]$ are given (classical) polynomials. 
\end{abstract}

\section{Introduction and preliminaries}

Perhaps G.~Ancochea was the first who studied additive mappings from a ring into another ring which also fulfill a `polynomial equation'. More concretely in \cite{Anc42} he described those additive functions that preserve squares. Later, these results were strengthened by (among others) Kaplansky
\cite{Kap} and Jacobson--Rickart \cite{JR}.

 Recall that if  $R, R'$ are rings, then the mapping $\varphi:R\rightarrow R'$ is called a \emph{homomorphism} if
\[
 \varphi(a+b)=\varphi(a)+\varphi(b)
\qquad \left(a, b\in R\right)
\]
and
\[
 \varphi(ab)=\varphi(a)\varphi(b)
\qquad \left(a, b\in R\right).
\]
Furthermore, the function $\varphi:R\to R'$ is an
\emph{anti-homomorphism} if
\[
 \varphi(a+b)=\varphi(a)+\varphi(b)
\qquad \left(a, b\in R\right)
\]
and
\[
 \varphi(ab)=\varphi(b)\varphi(a)
\qquad \left(a, b\in R\right).
\]

Henceforth, $\mathbb{N}$ will denote the set of the positive
integers. Let $n\in\mathbb{N}, n\geq 2$ be fixed. 
The function $\varphi:R\rightarrow R'$ is called an $n$-Jordan
homomorphism if
\[
\varphi(a+b)=\varphi(a)+\varphi(b) \qquad \left(a, b\in R\right)
\]
and
\[
 \varphi(a^{n})=\varphi(a)^{n}
\qquad \left(a\in R\right).
\]
In case $n=2$ we speak about homomorphisms
and Jordan homomorphisms, respectively. It was G.~Ancochea who
firstly dealt with the connection of Jordan homomorphisms and
homomorphisms, see \cite{Anc42}. These results were
generalized and extended in several ways, see for instance
\cite{JR}, \cite{Kap}, \cite{Zel68}. 
In \cite{Her} I.N.~Herstein showed that 
if $\varphi$ is a Jordan homomorphism of a ring $R$
\emph{onto} a prime ring $R'$ of characteristic different from $2$
and $3$, then either $\varphi$ is a homomorphism or an
anti-homomorphism. 

Besides homomorphisms, derivations also play a key role in the theory of rings and fields. Concerning this notion, we will follow \cite[Chapter 14]{Kuc}. 

Let $Q$ be a ring and let $P$ be a subring of $Q$.
A function $d\colon P\rightarrow Q$ is called a \emph{derivation}\index{derivation} if it is additive,
i.e. 
\[
d(x+y)=d(x)+d(y)
\quad
\left(x, y\in P\right)
\]
and also satisfies the so-called \emph{Leibniz rule}\index{Leibniz rule}, i.e.  equation
\[
d(xy)=d(x)y+xd(y)
\quad
\left(x, y\in P\right). 
\]

It is well-known that in case of additive functions, Hamel bases play an important role. 
As \cite[Theorem 14.2.1]{Kuc} shows in case of derivations, algebraic bases are fundamental. 

\begin{thm}\label{T14.2.1}
Let $(\mathbb{K}, +,\cdot)$ be a field of characteristic zero, let $(\mathbb{F}, +,\cdot)$
be a subfield of $(\mathbb{K}, +,\cdot)$, let $S$ be an algebraic base of $\mathbb{K}$ over $\mathbb{F}$,
if it exists, and let $S=\emptyset$ otherwise.
Let $f\colon \mathbb{F}\to \mathbb{K}$ be a derivation.
Then, for every function $u\colon S\to \mathbb{K}$,
there exists a unique derivation $g\colon \mathbb{K}\to \mathbb{K}$
such that $g \vert_{\mathbb{F}}=f$ and $g \vert_{S}=u$.
\end{thm}

\subsection{Generalized polynomial functions}

While proving our results, the so-called Polarization formula for multi-additive functions and the symmetrization method will play a key role. In this subsection the most important notations and statements are summarized. Here we follow the monograph \cite{Sze91}.

\begin{dfn}
 Let $G, S$ be commutative semigroups, $n\in \mathbb{N}$ and let $A\colon G^{n}\to S$ be a function.
 We say that $A$ is \emph{$n$-additive} if it is a homomorphism of $G$ into $S$ in each variable.
 If $n=1$ or $n=2$ then the function $A$ is simply termed to be \emph{additive}
 or \emph{bi-additive}, respectively.
\end{dfn}

The \emph{diagonalization} or \emph{trace} of an $n$-additive
function $A\colon G^{n}\to S$ is defined as
 \[
  A^{\ast}(x)=A\left(x, \ldots, x\right)
  \qquad
  \left(x\in G\right).
 \]
As a direct consequence of the definition each $n$-additive function
$A\colon G^{n}\to S$ satisfies
\[
 A(x_{1}, \ldots, x_{i-1}, kx_{i}, x_{i+1}, \ldots, x_n)
 =
 kA(x_{1}, \ldots, x_{i-1}, x_{i}, x_{i+1}, \ldots, x_{n})
 \qquad 
 \left(x_{1}, \ldots, x_{n}\in G\right)
\]
for all $i=1, \ldots, n$, where $k\in \mathbb{N}$ is arbitrary. The
same identity holds for any $k\in \mathbb{Z}$ provided that $G$ and
$S$ are groups, and for $k\in \mathbb{Q}$, provided that $G$ and $S$
are linear spaces over the rationals. For the diagonalization of $A$
we have
\[
 A^{\ast}(kx)=k^{n}A^{\ast}(x)
 \qquad
 \left(x\in G\right).
\]

The above notion can also be extended for the case $n=0$ by letting 
$G^{0}=G$ and by calling $0$-additive any constant function from $G$ to $S$. 

One of the most important theoretical results concerning
multiadditive functions is the so-called \emph{Polarization
formula}, that briefly expresses that every $n$-additive symmetric
function is \emph{uniquely} determined by its diagonalization under
some conditions on the domain as well as on the range. Suppose that
$G$ is a commutative semigroup and $S$ is a commutative group. The
action of the {\emph{difference operator}} $\Delta$ on a function
$f\colon G\to S$ is defined by the formula
\[\Delta_y f(x)=f(x+y)-f(x)
\qquad
\left(x, y\in G\right). \]
Note that the addition in the argument of the function is the
operation of the semigroup $G$ and the subtraction means the inverse
of the operation of the group $S$.

\begin{thm}[Polarization formula]\label{Thm_polarization}
 Suppose that $G$ is a commutative semigroup, $S$ is a commutative group, $n\in \mathbb{N}$.
 If $A\colon G^{n}\to S$ is a symmetric, $n$-additive function, then for all
 $x, y_{1}, \ldots, y_{m}\in G$ we have
 \[
  \Delta_{y_{1}, \ldots, y_{m}}A^{\ast}(x)=
  \left\{
  \begin{array}{rcl}
   0 & \text{ if} & m>n \\
   n!A(y_{1}, \ldots, y_{m}) & \text{ if}& m=n.
  \end{array}
  \right.
 \]

\end{thm}

\begin{cor}
 Suppose that $G$ is a commutative semigroup, $S$ is a commutative group, $n\in \mathbb{N}$.
 If $A\colon G^{n}\to S$ is a symmetric, $n$-additive function, then for all $x, y\in G$
 \[
  \Delta^{n}_{y}A^{\ast}(x)=n!A^{\ast}(y).
\]
\end{cor}

\begin{lem}
\label{mainfact}
  Let $n\in \mathbb{N}$ and suppose that the multiplication by $n!$ is surjective in the commutative semigroup $G$ or injective in the commutative group $S$. Then for any symmetric, $n$-additive function $A\colon G^{n}\to S$, $A^{\ast}\equiv 0$ implies that
  $A$ is identically zero, as well.
\end{lem}

\begin{dfn}
 Let $G$ and $S$ be commutative semigroups, a function $p\colon G\to S$ is called a \emph{generalized polynomial} from $G$ to $S$ if it has a representation as the sum of diagonalizations of symmetric multi-additive functions from $G$ to $S$. In other words, a function $p\colon G\to S$ is a 
 generalized polynomial if and only if, it has a representation 
 \[
  p= \sum_{k=0}^{n}A^{\ast}_{k}, 
 \]
where $n$ is a nonnegative integer and $A_{k}\colon G^{k}\to S$ is a symmetric, $k$-additive function for each 
$k=0, 1, \ldots, n$. In this case we also say that $p$ is a generalized polynomial \emph{of degree at most $n$}. 

Let $n$ be a nonnegative integer, functions $p_{n}\colon G\to S$ of the form 
\[
 p_{n}= A_{n}^{\ast}, 
\]
where $A_{n}\colon G^{n}\to S$ are the so-called \emph{generalized monomials of degree $n$}. 
\end{dfn}

\begin{rem}
 Obviously, generalized monomials 
of degree $0$ are constant functions and generalized monomials of degree $1$ are additive functions. 

 Furthermore, generalized monomials of degree $2$ will be termed \emph{quadratic functions}. 
\end{rem}

\subsection{Polynomial functions}

As Laczkovich \cite{Lac04} enlightens, on groups there are several polynomial notions. One of them is that we introduced in subsection 1.2, that is the notion of generalized polynomials. 
As we will see in the forthcoming sections, not only this notion, but also that of \emph{(normal) polynomials} will be important. The definitions and results recalled here can be found in \cite{Sze91}. 

Throughout this subsection $G$ is assumed to be  a commutative group (written additively).

\begin{dfn}
{\it Polynomials} are elements of the algebra generated by additive
functions over $G$. Namely, if $n$ is a positive integer,
$P\colon\mathbb{C}^{n}\to \mathbb{C}$ is a (classical) complex
polynomial in
 $n$ variables and $a_{k}\colon G\to \mathbb{C}\; (k=1, \ldots, n)$ are additive functions, then the function
 \[
  x\longmapsto P(a_{1}(x), \ldots, a_{n}(x))
 \]
is a polynomial and, also conversely, every polynomial can be
represented in such a form.
\end{dfn}

\begin{rem}
 For the sake of easier distinction, at some places polynomials will be called normal polynomials. 
\end{rem}

\begin{rem}
 We recall that the elements of $\mathbb{N}^{n}$ for any positive integer $n$ are called
 ($n$-dimensional) \emph{multi-indices}.
 Addition, multiplication and inequalities between multi-indices of the same dimension are defined component-wise.
 Further, we define $x^{\alpha}$ for any $n$-dimensional multi-index $\alpha$ and for any
 $x=(x_{1}, \ldots, x_{n})$ in $\mathbb{C}^{n}$ by
 \[
  x^{\alpha}=\prod_{i=1}^{n}x_{i}^{\alpha_{i}}
 \]
where we always adopt the convention $0^{0}=1$. We also use the
notation $\left|\alpha\right|= \alpha_{1}+\cdots+\alpha_{n}$. With
these notations any polynomial of degree at most $N$ on the
commutative semigroup $G$ has the form
\[
 p(x)= \sum_{\left|\alpha\right|\leq N}c_{\alpha}a(x)^{\alpha}
 \qquad
 \left(x\in G\right),
\]
where $c_{\alpha}\in \mathbb{C}$ and $a=(a_1, \dots, a_n) \colon
G\to \mathbb{C}^{n}$ is an additive function. Furthermore, the
\emph{homogeneous term of degree $k$} of $p$ is
\[
 \sum_{\left|\alpha\right|=k}c_{\alpha}a(x)^{\alpha} .
\]
\end{rem}

It is easy to see that 
each polynomial, that is, any function of the form 
\[
  x\longmapsto P(a_{1}(x), \ldots, a_{n}(x)), 
 \]
where $n$ is a positive integer,
$P\colon\mathbb{C}^{n}\to \mathbb{C}$ is a (classical) complex
polynomial in
 $n$ variables and $a_{k}\colon G\to \mathbb{C}\; (k=1, \ldots, n)$ are additive functions, is a generalized polynomial. The converse however is in general not true. A complex-valued generalized polynomial $p$ defined on a commutative group $G$ is a polynomial \emph{if and only if} its variety (the linear space spanned by its translates) is of \emph{finite} dimension. 
To make the situation more clear, here we also recall Theorem 13.4 from Sz\'ekelyhidi \cite{Sze14}. 
 
 \begin{thm}\label{thm_torsion}
  The torsion free rank of a commutative group is finite \emph{if and only if} every generalized polynomial on the group is a polynomial. 
 \end{thm}

In the next section the lemma below will be used, see \cite[Lemma 6]{Gse22}, too. 

\begin{lem}\label{lem_monomial_deg}
Let $k$ and $n$ be positive integers and $f \colon  \mathbb{F} \to \mathbb{C}$ be a generalized monomial of degree $n$, where $\mathbb{F}$ is assumed to be a field with $\mathrm{char}(\mathbb{F}) = 0$. Then the mapping
\[
\mathbb{F}\ni x \longmapsto f(x^{k})
\]
is a generalized monomial of degree $n \cdot k$.
\end{lem}

Henceforth, not only the notion of (exponential) polynomials, but also that of \emph{decomposable functions} will be used. The basics of this concept are due to 
Shulman \cite{Shu10}, besides this we heavily rely on the work of Laczkovich \cite{Lac19}. 
 
\begin{dfn}
 Let $G$ be a group and $n\in \mathbb{N}, n\geq 2$.
 A function $F\colon G^{n}\to \mathbb{C}$ is said to be
 \emph{decomposable} if it can be written as a finite sum of products
 $F_{1}\cdots F_{k}$, where all $F_{i}$ depend on disjoint sets of variables.
\end{dfn}

\begin{rem}
 Without loss of generality we can suppose that $k=2$ in the above definition, that is,
 decomposable functions are those mappings that can be written in the form
 \[
  F(x_{1}, \ldots, x_{n})= \sum_{E}\sum_{j}A_{j}^{E}B_{j}^{E}
 \]
where $E$ runs through all non-void proper subsets of $\left\{1,
\ldots, n\right\}$ and for each $E$ and $j$ the function $A_{j}^{E}$
depends only on variables $x_{i}$ with $i\in E$, while $B_{j}^{E}$
depends only on the variables $x_{i}$ with $i\notin E$.
\end{rem}

\begin{thm}\label{thm_decomposable}
Let $G$ be a commutative topological  semigroup with unit. 
A continuous function $f\colon G\to \mathbb{C}$ is a generalized exponential polynomial 
\emph{if and only if} there is a positive integer $n\geq 2$ such that the mapping
\[
G^{n} \ni (x_{1}, \ldots, x_{n}) \longmapsto f(x_1+\cdots+ x_n) 
\]
is decomposable. 
\end{thm}

The notion of derivations can be extended in several ways. We will employ the concept of higher order derivations according to Reich \cite{Rei98} and Unger--Reich \cite{UngRei98}. For further results on characterization theorems on higher order derivations consult e.g. \cite{Eba15, Eba17, EbaRieSah} and 
\cite{GseKisVin18}. 

\begin{dfn}
 Let $\mathbb{F}\subset \mathbb{C}$ be a field. The identically zero map is the only \emph{derivation of order zero}. For each $n\in \mathbb{N}$, an additive mapping 
 $f\colon \mathbb{F}\to \mathbb{C}$ is termed to be a \emph{derivation of order $n$}, if there exists $B\colon \mathbb{F}\times \mathbb{F}\to \mathbb{C}$ such that 
 $B$ is a bi-derivation of order $n-1$ (that is, $B$ is a derivation of order $n-1$ in each variable) and 
 \[
  f(xy)-xf(y)-f(x)y=B(x, y) 
  \qquad 
  \left(x, y\in \mathbb{F}\right). 
 \]
 The set of derivations of order $n$ of the field $\mathbb{F}$ will be denoted by $\mathscr{D}_{n}(\mathbb{F})$. 
\end{dfn}

\begin{rem}
\label{pathologic}
Since $\mathscr{D}_{0}(\mathbb{F})=\{0\}$, the only bi-derivation of order zero is the identically zero function, thus $f\in \mathscr{D}_{1}(\mathbb{F})$ if and only if
  \[
   f(xy)=xf(y)+f(x)y 
   \qquad 
   \left(x, y\in \mathbb{F}\right), 
  \]
that is, the notions of first order derivations and derivations coincide. On the other hand for any $n\in \mathbb{N}$ the set $\mathscr{D}_{n}(\mathbb{F})\setminus \mathscr{D}_{n-1}(\mathbb{F})$ is nonempty because  $d_{1}\circ \cdots \circ d_{n}\in \mathscr{D}_{n}(\mathbb{F})$, but $d_{1}\circ \cdots \circ d_{n}\notin \mathscr{D}_{n-1}(\mathbb{F})$, where $d_{1}, \ldots, d_{n}\in \mathscr{D}_{1}(\mathbb{F})$ are non-identically zero derivations.
\end{rem}

The main result of \cite{KisLac18} is Theorem 1.1 that reads in our settings as follows. 

\begin{thm}\label{thm_KisLac}
 Let $\mathbb{F}\subset \mathbb{C}$ be a field and let $n$ be a positive integer. Then, for every function $D \colon \mathbb{F}\to \mathbb{C}$, 
$D\in \mathscr{D}_{n}(\mathbb{F})$ if and only if 
$D$ is additive on $\mathbb{F}$, $D(1) = 0$, and $\dfrac{D}{\mathrm{id}}$, as a map from the group $\mathbb{F}^{\times}$ to $\mathbb{C}$, is a generalized polynomial of degree at most $n$. Here $\mathrm{id}$ stands for the identity map defined on $\mathbb{F}$.
\end{thm}

\section{Results}

Henceforth let $\mathbb{F}\subset \mathbb{C}$ be a field, $n$ be a positive integer and $P\in \mathbb{F}[x]$ be a polynomial. In what follows we will study generalized monomials $f\colon \mathbb {F}\to \mathbb{C}$ of degree $n$  under the condition that the mapping 
\[
\mathbb{F}\ni x\longmapsto f(P(x))
\]
is a (normal) polynomial. 

At first we show that instead of polynomials $P$, we always may restrict ourselves to (classical) monomials. For this we need the following statement which is in some sense an extension of Lemma \ref{lem_monomial_deg}.

\begin{lem}\label{lem3}
Let $\mathbb{F}\subset \mathbb{C}$ be a field,  $n\in \mathbb{N}$, $\alpha_{1}, \ldots, \alpha_{n}$ be non-negative integers and 
$F\colon \mathbb{F}\to \mathbb{C}$ be a symmetric and $n$-additive function. Then the function $g\colon \mathbb{F}\to \mathbb{C}$ defined by 
\[
g(x)= F(x^{\alpha_{1}}, \ldots, x^{\alpha_{n}}) 
\qquad 
\left(x\in \mathbb{F}\right)
\]
is a generalized monomial of degree $(\alpha_{1}+\cdots+ \alpha_{n})$. 
\end{lem}

\begin{proof}
Let $\mathbb{F}\subset \mathbb{C}$ be a field,  $n\in \mathbb{N}$, $\alpha_{1}, \ldots, \alpha_{n}$ be non-negative integers and 
$F\colon \mathbb{F}\to \mathbb{C}$ be a symmetric and $n$-additive function and consider the function $g\colon \mathbb{F}\to \mathbb{C}$ defined by 
\[
g(x)= F(x^{\alpha_{1}}, \ldots, x^{\alpha_{n}}) 
\qquad 
\left(x\in \mathbb{F}\right). 
\]
Let $N=\alpha_{1}+\cdots+ \alpha_{n}$ and define the mapping $\mathscr{F}\colon \mathbb{F}^{N}\to \mathbb{C}$ through 
\[
\mathscr{F}(x_{1}, \ldots, x_{N})= \dfrac{1}{N!}\sum_{\sigma \in \mathscr{S}_{n}}
F(x_{\sigma(1)}\cdots x_{\sigma(\alpha_{1})}, \ldots, x_{\sigma(N-\alpha_{n}+1}\cdots x_{\sigma(N)})
\qquad 
\left(x_{1}, \ldots, x_{N}\in \mathbb{F}\right). 
\]
Since $F$ is a symmetric and $n$-additive function, $\mathscr{F}$ is also symmetric and $N$-additive, further we have 
\[
\mathscr{F}(x_{1}, \ldots, x_{N})= 
F(x^{\alpha_{1}}, \ldots, x^{\alpha_{n}})= g(x)
\qquad 
\left(x\in \mathbb{F}\right). 
\]
Thus $g$ can be represented as the trace of a symmetric and $(\alpha_{1}+\cdots+ \alpha_{n})$-additive mapping, showing that the function $g$ is indeed a generalized monomial of degree $(\alpha_{1}+\cdots+ \alpha_{n})$. 
\end{proof}

\begin{lem}\label{lem4}
Let $k, n\in \mathbb{N}$, $k\geq 2$,  $\mathbb{F}\subset \mathbb{C}$ be a field, $P\in \mathbb{F}[x]$ be a (classical) polynomial of degree $k$ with leading coefficient $1$ and $f\colon \mathbb{F}\to \mathbb{C}$ be a generalized monomial of degree $n$. If the mapping 
\[
\mathbb{F}\ni x \longmapsto f(P(x))
\]
is a normal polynomial, then the mapping 
\[
\mathbb{F} \ni x \longmapsto f(x^{k})
\]
is a normal polynomial as well. 
\end{lem}

\begin{proof}
Let $k, n\in \mathbb{N}$, $\mathbb{F}\subset \mathbb{C}$ be a field, $P\in \mathbb{F}[x]$ be a (classical) polynomial of degree $k$ and $f\colon \mathbb{F}\to \mathbb{C}$ be a generalized monomial of degree $n$. Suppose further that the mapping 
\[
\mathbb{F}\ni x \longmapsto f(x^{k})
\]
is a normal polynomial. 

Since $P\in \mathbb{F}[x]$ is a (classical) polynomial of degree $k$, we have 
\[
P(x)= \sum_{l=0}^{k}\alpha_{l}x^{l} 
\qquad 
\left(x\in \mathbb{F}\right), 
\]
with some constants $\alpha_{l}\in \mathbb{F}$, $l=0, 1, \ldots, k$ such that $\alpha_{k}=1$. Further, as $f\colon \mathbb{F}\to \mathbb{C}$ is a monomial of degree $n$, there exists a uniquely determined symmetric and $n$-additive function $F\colon \mathbb{F}^{n}\to \mathbb{C}$ such that 
\[
f(x)= F(x, \ldots, x) 
\qquad 
\left(x\in \mathbb{F}\right). 
\]
This together yield that there exists a positive integer $m$, linearly independent additive functions $a_{1}, \ldots, a_{m}$ and a complex polynomial $Q\in \mathbb{C}[x]$ such that 
\[
F\left(\sum_{l=0}^{k}\alpha_{l}x^{l}, \ldots, \sum_{l=0}^{k}\alpha_{l}x^{l}\right)=
P(a_{1}(x), \ldots, a_{m}(x))= \sum_{\substack{\beta\in \mathbb{N}^{m}\\ |\beta|\leq kn}}c_{\beta}a(x)^{\beta}
\qquad 
\left(x\in \mathbb{F}\right). 
\]
Using that $F$ is symmetric and additive in each of its variables, we get that 
\[
F(x^{k}, \ldots, x^{k})+ 
\lambda_{k-1, k, \ldots, k} F(\alpha_{k-1}x^{k-1}, \alpha_{k}x^{k}, \ldots, \alpha_{k}x^{k})
+ \cdots+
\lambda_{0, \ldots, 0}F(\alpha_{0}, \ldots, \alpha_{0})= \sum_{\substack{\beta\in \mathbb{N}^{m}\\ |\beta|\leq kn}}c_{\beta}a(x)^{\beta}
\]
for all $x\in \mathbb{F}$ with some complex constants $\lambda_{\gamma}$, $\gamma \in \mathbb{N}^{n}$, $|\gamma|\leq n$. Due to Lemma \ref{lem3}, the left and also the right hand side of this identity is a generalized polynomial (being linear combinations of generalized monomials) and these agree for all possible values. This can however happen only if the monomial terms are the same on each side. From this we get especially that 
\[
F(x^{k}, \ldots, x^{k}) = \sum_{\substack{\beta\in \mathbb{N}^{m}\\ |\beta|=kn}}c_{\beta}a(x)^{\beta}, 
\]
showing that the mapping 
\[
\mathbb{F} \ni x \longmapsto f(x^{k})
\]
is a normal polynomial as well. 
\end{proof}

\begin{rem}
According to the previous lemma once $f\colon \mathbb{F}\to \mathbb{C}$ is a generalized monomial of degree $n$ such that  the mapping 
\[
\mathbb{F}\ni x \longmapsto f(P(x))
\]
is a normal polynomial, then the mapping 
\[
\mathbb{F} \ni x \longmapsto f(x^{k})
\]
is a normal polynomial as well. This enables us to restrict ourselves to considering generalized monomials for which the mapping 
\[
\mathbb{F} \ni x \longmapsto f(x^{k})
\]
is a normal polynomial for a fixed $k\geq 2$. At the same time, we have to emphasize that the assumption that the mapping 
\[
\mathbb{F}\ni x \longmapsto f(P(x))
\]
is a normal polynomial for a fixed polynomial $P\in \mathbb{F}[x]$ is more restrictive than the previous one. To illustrate this let us consider the following example. Let $f\colon \mathbb{F}\to \mathbb{C}$ be a quadratic function and let 
\[
P(x)= x^{2}+\alpha_{1}x+\alpha_{0} 
\qquad 
\left(x\in \mathbb{F}\right), 
\]
where $\alpha_{1}, \alpha_{0}\in \mathbb{F}$ are fixed. The assumption that the mapping 
\[
\mathbb{F}\ni x \longmapsto f(P(x))
\]
is a normal polynomial means that we have 
\[
F(x^{2}+\alpha_{1}x+\alpha_{0}, x^{2}+\alpha_{1}x+\alpha_{0})= Q(a_{1}(x), \ldots, a_{m}(x)) 
\qquad 
\left(x\in \mathbb{F}\right)
\]
with appropriate linearly independent additive functions $a_{1}, \ldots, a_{m}\colon \mathbb{F}\to \mathbb{C}$ and a complex polynomial $Q\in \mathbb{C}[x_{1}, \ldots, x_{m}]$, where $F\colon \mathbb{F}^{2}\to \mathbb{C}$ denotes the uniquely determined bi-additive mapping whose trace is the quadratic function $f$. Using the symmetry and also the bi-additivity of $F$, we get that 
\[
F(x^{2}, x^{2})+ 2F(x^{2}, \alpha_{1}x)+ 2F(x^2, \alpha_{0})+ F(\alpha_{1}x, \alpha_{1}x)+2F(\alpha_{1}x, \alpha_{0})+ F(\alpha_{0}, \alpha_{0})= 
Q(a_{1}(x), \ldots, a_{m}(x)) 
\qquad 
\left(x\in \mathbb{F}\right). 
\]
Due to Lemma \ref{lem3}, the mappings 
\[
\mathbb{F}\ni x \longmapsto F(x^{i}, x^{j})
\]
are generalized monomials of degree $(i+j)$. Further, generalized monomials of different degree are linearly independent. Therefore, not only the mapping 
\[
\mathbb{F}\ni x \longmapsto F(x^{2}, x^{2})
\]
is a normal polynomial, but also the mappings 
\[
\begin{array}{rcl}
\mathbb{F}\ni &x& \longmapsto F(x^{2}, \alpha_{1}x)\\
\mathbb{F}\ni &x& \longmapsto 2F(x^{2}, \alpha_{0})+F(\alpha_{1}x, \alpha_{1}x)\\
\mathbb{F}\ni &x& \longmapsto F(\alpha_{1}x, \alpha_{0})\\
\mathbb{F}\ni &x& \longmapsto F(\alpha_{0}, \alpha_{0}), 
\end{array}
\]
too. Obviously, this is always true for the last two functions. Nevertheless, the fact that the first three functions have this property carries more information than that only the first of them is a normal polynomial. 
\end{rem}

Our second lemma says that while considering this problem, we may restrict ourselves to homogeneous (normal) polynomials. 

\begin{lem}\label{lemma_hom}
Let $k, n\in \mathbb{N}$, $k\geq 2$,  $\mathbb{F}\subset \mathbb{C}$ be a field and $f\colon \mathbb{F}\to \mathbb{C}$ be a generalized monomial of degree $n$. If the mapping 
\[
\mathbb{F}\ni x \longmapsto f(x^{k})
\]
is a normal polynomial, then there exists a \emph{homogeneous} complex polynomial $\tilde{P}$ and there are linearly independent, complex valued additive functions $a_{1}, \ldots, a_{m}$ on $\mathbb{F}$ such that 
\[
f(x^{k})= \tilde{P}(a_{1}(x), \ldots, a_{m}(x)) 
\qquad 
\left(x\in \mathbb{F}\right), 
\]
in other words, we have 
\[
f(x^{k})= \sum_{\substack{\alpha\in \mathbb{N}^{m}\\ |\alpha|=kn}}\lambda_{\alpha}a^{\alpha}(x)
=
\sum_{\substack{\alpha_{1}, \ldots, \alpha_{m}\geq 0\\ 
\alpha_{1}+\cdots+ \alpha_{m}=kn}}\lambda_{\alpha_{1}, \ldots, \alpha_{m}}a_{1}^{\alpha_{1}}(x)\cdots a_{m}^{\alpha_{m}}(x)
\]
for each $x\in \mathbb{F}$. 
\end{lem}

\begin{proof}
Let $k, n\in \mathbb{N}$, $k\geq 2$,  $\mathbb{F}$ be a field and $f\colon \mathbb{F}\to \mathbb{C}$ be a generalized monomial of degree $n$. Assume further that  the mapping 
\[
\mathbb{F}\ni x \longmapsto f(x^{k})
\]
is a normal polynomial. Then due to Lemma \ref{lem_monomial_deg} this mapping is a generalized monomial of degree $kn$ and hence it is $\mathbb{Q}$-homogeneous of degree $kn$. 

If additionally, this mapping is a normal polynomial, there exists a complex polynomial $P$ and there are linearly independent, complex valued additive functions $a_{1}, \ldots, a_{m}$ on $\mathbb{F}$ such that 
\[
f(x^{k})= P(a_{1}(x), \ldots, a_{m}(x))= P(a(x))= \sum_{l=0}^{N}\sum_{\substack{\alpha\in \mathbb{N}^{k}\\ |\alpha|=l}} \lambda_{\alpha} a^{\alpha}(x)
\qquad 
\left(x\in \mathbb{F}\right). 
\]
Let $r\in \mathbb{Q}$ be arbitrary and let us substitute $rx$ in place of $x$ in the above identity. Using the $\mathbb{Q}$-homogeneity of the involved functions, we deduce 
\[
r^{kn}f(x^{k})-
 \sum_{l=0}^{N} r^{l}\sum_{\substack{\alpha\in \mathbb{N}^{k}\\ |\alpha|=l}} \lambda_{\alpha}  a^{\alpha}(x)=0
\qquad 
\left(r\in \mathbb{Q}, x\in \mathbb{F}\right). 
\]
Observe that the left hand side of this identity is a polynomial in $r$ that is identically zero. So all of its coefficients should vanish, yielding that 
\[
f(x^{k})= \sum_{\substack{\alpha\in \mathbb{N}^{m}\\ |\alpha|=kn}}\lambda_{\alpha}a^{\alpha}(x)
\]
holds for all $x\in \mathbb{F}$. 
\end{proof}

At first glance the assumption of the lemma above, i.e., the mapping 
\[
\mathbb{F} \ni x\longmapsto f(x^{k})
\]
is a normal polynomial, seems a bit artificial. Nevertheless, the following examples show that this is not the case. 

\begin{ex}
Let $k$ be a positive integer, $\varphi_{1}, \ldots, \varphi_{k}\colon \mathbb{F}\to \mathbb{C}$ be linearly independent homomorphisms and $\lambda_{i, j}\in \mathbb{C}$ for all $i, j=1, \ldots, k$. 
Then the mapping $f\colon \mathbb{F}\to \mathbb{C}$ defined by 
\[
f(x)= \sum_{i, j=1}^{k}\lambda_{i, j}\varphi_{i}(x)\varphi_{j}(x) 
\qquad 
\left(x\in \mathbb{F}\right)
\]
is a quadratic function. Further if $n\in \mathbb{N}$, then we also have 
\[
f(x^{n})= \sum_{i, j=1}^{k}\lambda_{i, j}\varphi_{i}(x)^{n}\varphi_{j}(x)^{n}
\qquad 
\left(x\in \mathbb{F}\right). 
\]
In other words, we have 
\[
f(x^{n})= P(\varphi_{1}(x), \ldots, \varphi_{k}(x))
\qquad 
\left(x\in \mathbb{F}\right), 
\]
where the $k$-variable complex, homogeneous polynomial $P$ is defined by 
\[
P(x_{1}, \ldots, x_{k})= \sum_{i, j=1}^{k}\lambda_{i, j}x_{i}^{n}x_{j}^{n}
\qquad 
\left(x_{1}, \ldots, x_{k}\in \mathbb{C}\right). 
\]
\end{ex}

\begin{ex}
Suppose now that $\mathbb{F}$ is a subfield of $\mathbb{C}$. 
Let $k$ be a positive integer, $d_{1}, \ldots, d_{k}\colon \mathbb{F}\to \mathbb{C}$ be linearly independent derivations and $\lambda_{i, j}\in \mathbb{C}$ for all $i, j=1, \ldots, k$. 
Then the mapping $f\colon \mathbb{F}\to \mathbb{C}$ defined by 
\[
f(x)= \sum_{i, j=1}^{n}\lambda_{i, j}d_{i}(x)d_{j}(x) 
\qquad 
\left(x\in \mathbb{F}\right)
\]
is a quadratic function. Further if $n\in \mathbb{N}$, then we also have 
\[
f(x^{n})= \sum_{i, j=1}^{n}\lambda_{i, j}n^{2}x^{2n-2}d_{i}(x)d_{j}(x) 
\qquad 
\left(x\in \mathbb{F}\right). 
\]
In other words, 
\[
f(x^{n})= P(x, d_{1}(x), \ldots, d_{k}(x)) 
\qquad 
\left(x\in \mathbb{F}\right), 
\]
where the $(k+1)$-variable complex polynomial $P$ is defined by 
\[
P(x_{1}, \ldots, x_{k}, z)= \sum_{i, j=1}^{n}\lambda_{i, j}n^{2}z^{2n-2}x_{i}x_{j} 
\qquad 
\left(x_{1}, \ldots, x_{k}, z\in \mathbb{C}\right). 
\]
\end{ex}

\begin{ex}
Assume in this example that $\mathbb{F}$ is a subfield of $\mathbb{C}$. 
Let $k$ be a positive integer and $d\colon \mathbb{F}\to \mathbb{C}$ be a derivation and define the quadratic function $f\colon \mathbb{F}\to \mathbb{C}$ by 
\[
f(x)= d^{k}(x^{2})= \underbrace{d\circ \cdots \circ d}_\text{$k$ times}(x^{2})  
\qquad 
\left(x\in \mathbb{C}\right). 
\]
Further, if $n$ is a positive integer, then we have 
\[
f(x^{n})= d^{k}(x^{2n})= 
\sum_{\substack{l_{1}, \ldots, l_{2n}\geq 0\\ l_{1}+\cdots+l_{2n}=k}}
\binom{k}{l_{1}, \ldots, l_{2n}}d^{l_{1}}(x)\cdots d^{l_{2n}}(x)
\]
for all $x\in \mathbb{F}$. In other words, we have 
\[
f(x^{n})= P(x, d(x), d\circ d(x), \ldots, d^{k}(x)) 
\qquad 
\left(x\in \mathbb{F}\right). 
\]
\end{ex}

\begin{rem}
All of the above examples can easily be generalized from quadratic functions to monomial functions.  To get similar examples instead of quadratic functions, for monomials of degree $n$, where $n$ is a fixed it is enough to consider the mappings 
\[
f(x)= \sum_{i=1}^{k}\sum_{\substack{\alpha\in \mathbb{N}^{n}\\ |\alpha|=n}}\lambda_{i}\Phi_{i}(x)^{\alpha}
\qquad 
\left(x\in \mathbb{F}\right), 
\]
\[
f(x)= \sum_{i=1}^{k}\sum_{\substack{\alpha\in \mathbb{N}^{n}\\ |\alpha|=n}}\lambda_{i}D _{i}(x)^{\alpha}
\qquad 
\left(x\in \mathbb{F}\right)
\]
and 
\[
f(x)= d^{k}(x^{n})= \underbrace{d\circ \cdots \circ d}_\text{$k$ times}(x^{n})
\qquad 
\left(x\in \mathbb{F}\right), 
\]
here 
\[
\Phi_{i}(x)= \left(\varphi_{1, i}(x), \ldots, \varphi_{n, i}(x)\right)
\quad 
\text{and}
\quad 
D_{i}(x)=\left(d_{1, i}(x), \ldots, d_{n, i}(x)\right)
\qquad 
\left(x\in \mathbb{F}\right), 
\]
where the functions $\varphi_{l, i}$ are homomorphisms, while $d$ and $d_{l, i}$ are derivations for all possible indices $l$ and $i$. 
\end{rem}

Regarding `polynomial equations' for generalized monomials, we note that in the literature, there are several results for additive functions $a\colon \mathbb{F}\to \mathbb{C}$ that also satisfy a polynomial equation. Based on the results presented above, the following statement can be deduced. 

\begin{prop}
Let $\mathbb{F}\subset \mathbb{C}$ be a field, $k\in \mathbb{N}$, $k\geq 2$ and $P\in \mathbb{Q}[x]$ be a (classical) polynomial of degree $k$. 
\begin{enumerate}[(i)]
    \item If the additive function $a\colon \mathbb{F}\to \mathbb{C}$ fulfills 
    \[
    a(P(x))= P(a(x)) 
    \qquad 
    \left(x\in \mathbb{F}\right)
    \]
    then there exists a homomorphism $\varphi\colon \mathbb{F}\to \mathbb{C}$ such that $a(x)= a(1)\varphi(x)$ for all $x\in \mathbb{F}$. Further, we also have $a(1)\in \left\{ 0, 1\right\}$. 
    \item If the additive function $a\colon \mathbb{F}\to \mathbb{C}$ fulfills 
    \[
    a(P(x))= P'(x)a(x) 
    \qquad 
    \left(x\in \mathbb{F}\right), 
    \]
    then $a$ is a derivation. 
\end{enumerate}
\end{prop}

\begin{proof}
\begin{enumerate}[(i)]
    \item Let $\mathbb{F}\subset \mathbb{C}$ be a field, $k\in \mathbb{N}$, $k\geq 2$ and $P\in \mathbb{Q}[x]$ be a (classical) polynomial of degree $k$. Suppose further that the additive function $a\colon \mathbb{F}\to \mathbb{C}$ fulfills 
    \[
    a(P(x))= P(a(x)) 
    \qquad 
    \left(x\in \mathbb{F}\right). 
    \]
    In other words, we have 
    \[
    a\left(\sum_{i=0}^{k}\alpha_{k}x^{k}\right)= \sum_{i=0}^{k}\alpha_{k}a(x)^{k}
    \]
    for all $x\in \mathbb{F}$ with some rational numbers $\alpha_{k}, \ldots, \alpha_{0}$. 
    Observe that this especially yields that the mapping 
    \[
    \mathbb{F} \ni x \longmapsto  a\left(\sum_{i=0}^{k}\alpha_{k}x^{k}\right)
    \]
    is a normal polynomial. Due to Lemma \ref{lem4} we infer that then the mapping 
    \[
    \mathbb{F} \ni x \longmapsto a(x^{k})
    \]
    is also a normal polynomial, where the $\mathbb{Q}$-homogeneity of $a$ was used, too. However, then necessarily 
    \[
    a(x^{k})= a(x)^{k} 
    \qquad 
    \left(x\in \mathbb{F}\right)
    \]
    holds. From this, we deduce (e.g. using the results of \cite{Her}) that there exists a homomorphism $\varphi\colon \mathbb{F}\to \mathbb{C}$ such that $a(x)= a(1)\varphi(x)$ for all $x\in \mathbb{F}$. Substituting this back to our equation, we finally get that the only possibility is that $a(1)\in \left\{ 0, 1\right\}$. 

    \item Using a similar reasoning that we did in case (i), here we deduce that the assumptions imply that necessarily 
    \[
    a(x^{k})= kx^{k-1}a(x) 
    \qquad 
    \left(x\in \mathbb{F}\right)
    \]
    holds for the additive function $a\colon \mathbb{F}\to \mathbb{C}$. Using some classical characterization theorems concerning derivations (for instance the results of \cite[Chapter 14]{Kuc}), we finally get that $a$ is indeed a derivation. 
\end{enumerate}
\end{proof}

\begin{rem}
We emphasize that the results of the previous statement are classical ones. 
Nevertheless, we would like to indicate that from one hand the problem we would like to consider in this paper has some prior results both in algebra and the theory of functional equations. Further, with the help of the statements presented here, the proofs can be significantly simplified (at least for mappings $a\colon \mathbb{F}\to \mathbb{C}$). 

In the papers \cite{Eba15, Eba17, EbaRieSah, GseKisVin18, GseKisVin19} further results can be found concerning additive functions that also fulfill certain polynomial equations. 
\end{rem}

\begin{rem}
We also note that related problems have already been considered by Z.~Boros and E. Garda–Mátyás in \cite{BorGar18, BorGar22} by Z.~Boros and R.~Menzer in \cite{BorMen22} and also by M. Amou in \cite{Amo20}. In these papers the authors consider real monomial functions,  which satisfy certain conditional equations on a specified planar curve. 
In \cite{Gse22}, the polynomial equation $f(P(x))=Q(f(x))$ for the monomial function $f\colon \mathbb{F}\to \mathbb{C}$ was considered. 
\end{rem}

The simplest special case of the problem we are interested in is when the generalized monomial $f\colon \mathbb{F}\to \mathbb{C}$ is of degree $1$, i.e., when $f$ is an additive function. In this regard, we have the following statement.

\begin{thm}
Let $k$ be a positive integer, $\mathbb{F}\subset \mathbb{C}$ be a field and $a\colon \mathbb{F}\to \mathbb{C}$ be an additive function. If the mapping 
\[
\mathbb{F} \ni x \longmapsto a(x^{k})
\]
is a (normal) polynomial, then $a$ is a generalized higher order derivation. 
\end{thm}

\begin{proof}
Let $k$ be a positive integer, $\mathbb{F}\subset \mathbb{C}$ be a field and $a\colon \mathbb{F}\to \mathbb{C}$ be an additive function. Suppose further that the mapping 
\[
\mathbb{F} \ni x \longmapsto a(x^{k})
\]
is a (normal) polynomial. Due to Lemma \ref{lemma_hom}, we can assume that this mapping is a \emph{homogeneous} (normal) polynomial, that is, we have 
\[
a(x^{k})
=
\sum_{\substack{\alpha_{1}, \ldots, \alpha_{m}\geq 0\\ 
\alpha_{1}+\cdots+ \alpha_{m}=k}}\lambda_{\alpha_{1}, \ldots, \alpha_{m}}a_{1}^{\alpha_{1}}(x)\cdots a_{m}^{\alpha_{m}}(x)
\]
for all $x\in \mathbb{F}$. Observe that both the sides of this identity is the trace of a symmetric and $k$-additive mapping. Indeed, the left hand side is the trace of the symmetric $k$-additive function 
\[
A(x_{1}, \ldots, x_{k})= a(x_{1}\cdots x_{k})
\qquad 
\left(x_{1}, \ldots, x_{k}\in \mathbb{F}\right)
\]
while the right hand side is the trace of the symmetric $k$-additive mapping 
\begin{multline*}
\widetilde{A}(x_{1}, \ldots, x_{k})\\
= 
\dfrac{1}{k!}\sum_{\sigma \in \mathcal{S}_{k}}
\sum_{\substack{\alpha_{1}, \ldots, \alpha_{m}\geq 0\\ 
\alpha_{1}+\cdots+ \alpha_{m}=k}}\lambda_{\alpha_{1}, \ldots, \alpha_{m}}
a_{1}(x_{\sigma(1)}) \cdots a_{1}(x_{\sigma(\alpha_{1})})
\cdots 
a_{m}(x_{\sigma(k-\alpha_{m}+1)}) \cdots a_{m}(x_{\sigma(k)})
\\
\left(x_{1}, \ldots, x_{k}\in \mathbb{F}\right). 
\end{multline*}
Therefore we have 
\[
    a(x_{1}\cdots x_{k})
  =
  \dfrac{1}{k!}\sum_{\sigma \in \mathcal{S}_{k}}
\sum_{\substack{\alpha_{1}, \ldots, \alpha_{m}\geq 0\\ 
\alpha_{1}+\cdots+ \alpha_{m}=k}}\lambda_{\alpha_{1}, \ldots, \alpha_{m}}
a_{1}(x_{\sigma(1)}) \cdots a_{1}(x_{\sigma(\alpha_{1})})
\cdots 
a_{m}(x_{\sigma(k-\alpha_{m}+1)}) \cdots a_{m}(x_{\sigma(k)})  
\]
for all $x_{1}, \ldots, x_{k}\in \mathbb{F}$. 

In other words, the mapping 
\[
\mathbb{F^{\times}}^{k} \ni (x_{1}, \ldots, x_{k})  \longmapsto a(x_{1}\cdots x_{k})
\]
is decomposable. Thus from Theorem \ref{thm_decomposable} we deduce that $a\colon \mathbb{F}^{\times}\to \mathbb{C}$ is a generalized exponential polynomial on the multiplicative group $\mathbb{F}^{\times}$ corresponding to the identity function, as exponential. Using Theorem \ref{thm_KisLac}, we finally obtain that $a$ is a generalized higher order derivation. 
\end{proof}

\begin{thm}\label{thm7}
Let $n\in \mathbb{N}$, $n\geq 2$ and $\mathbb{F}\subset \mathbb
C$ be a field. Assume that $f\colon \mathbb{F}\to \mathbb{C}$
is a quadratic function, while $a\colon \mathbb{F}\to \mathbb{C}$ is additive and we have 
\begin{equation}\label{Eq_1}
f(x^{n})= a(x)^{2n} 
\qquad 
\left(x\in \mathbb{F}\right). 
\end{equation}
Then there exists a complex constant $\alpha\in \mathbb{C}$ and a homomorphism $\varphi\colon \mathbb{F}\to \mathbb{C}$ such that 
\[
a(1)= \alpha \varphi(x) 
\qquad 
\text{and}
\qquad 
f(x)= \alpha^{2n}\varphi(x)^{2} 
\qquad 
\left(x\in \mathbb{F}\right). 
\]
And also conversely, if we define the functions $a$ and $f$ through the above formula then they satisfy equation \eqref{Eq_1} for all $x\in \mathbb{F}$. 
\end{thm}

\begin{proof}
Since $f$ is quadratic, there exists a uniquely determined symmetric and bi-additive mapping $F\colon \mathbb{F}^{2}\to \mathbb{C}$ such that 
\[
F(x, x)= f(x) 
\qquad 
\left(x\in \mathbb{F}\right). 
\]
Define the mapping $E\colon \mathbb{F}^{2n}\to \mathbb{C}$ by 
\[
E(x_{1}, \ldots, x_{2n})= \frac{1}{(2n)!} \sum_{\sigma\in \mathscr{S}_{2n}}
F(x_{\sigma(1)}\cdots x_{\sigma(n)}, x_{\sigma(n+1)} \cdots x_{\sigma(2n)})- 
a(x_{1})\cdots a(x_{2n}) 
\qquad 
\left(x_{1}, \ldots, x_{2n}\in \mathbb{F}\right), 
\]
where $\mathscr{S}_{2n}$ denotes the symmetric group of order $2n$. 
Since $F$ is bi-additive and $a$ is additive, the mapping $E$ is a symmetric and $2n$-additive mapping. Further its trace is 
\[
E(x, \ldots, x)= F(x^{n}, x^{n})-a(x)^{2n}= f(x^{n})-a(x)^{2n}=0 
\qquad 
\left(x\in \mathbb{F}\right), 
\]
due to the above equation. At the same time, its trace uniquely determines $E$. Thus the function $E$ vanishes identically, that is, 
\begin{equation}\label{symm}
\frac{1}{(2n)!} \sum_{\sigma\in \mathscr{S}_{2n}}
F(x_{\sigma(1)}\cdots x_{\sigma(n)}, x_{\sigma(n+1)} \cdots x_{\sigma(2n)})- 
a(x_{1})\cdots a(x_{2n})=0
\end{equation}
holds for all $x_{1}, \ldots, x_{2n}\in \mathbb{F}$. With the substitution 
$x_{i}=1$ for  $i=1, \ldots, 2n$ we get that 
\[
F(1, 1)-a(1)^{2n}= f(1)-a(1)^{2n}=0. 
\]
Let now $x\in \mathbb{F}$ be arbitrary and let 
\[
x_{1}= x \qquad 
\text{and}
\qquad 
x_{i}=1 
\qquad 
\text{for}
\qquad 
i=2, \ldots, 2n
\]
to deduce that 
\[
\mu_{1}F(x, 1)-\mu_{2}a(x)=0 
\qquad 
\left(x\in \mathbb{F}\right), 
\]
that is 
\begin{equation}\label{sym1}
F(x, 1)= \lambda a(x) 
\qquad 
\left(x\in \mathbb{F}\right)
\end{equation}
with an appropriate complex constant $\lambda$. 

Let again $x\in \mathbb{F}$ be arbitrary and let now 
\[
x_{1}= x, \; x_{2}= x \qquad 
\text{and}
\qquad 
x_{i}=0
\qquad 
\text{for}
\qquad 
i=3, \ldots, 2n
\]
in \eqref{symm} to obtain 
\[
\nu_{1}F(x, x)+\nu_{2}F(x^{2}, 1)+\nu_{3}a(x)^{2}=0
\qquad 
\left(x\in \mathbb{F}\right). 
\]
The latter identity together with \eqref{sym1} yields that there exists complex constants $\alpha_{1}, \alpha_{2}$ such that 
\[
F(x, x)= \alpha_{1}a(x^{2})+\alpha_{2}a(x)^{2} 
\qquad 
\left(x\in \mathbb{F}\right), 
\]
in other words, we have 
\[
f(x)= \alpha_{1}a(x^{2})+\alpha_{2}a(x)^{2} 
\qquad 
\left(x\in \mathbb{F}\right). 
\]
If we substitute this representation back to \eqref{Eq_1}, we obtain that 
\[
\alpha_{1}a(x^{2n})+\alpha_{2}a(x^{n})^{2}= a(x)^{2n}
\]
for all $x\in \mathbb{F}$. According to Theorem 12 of \cite{GseKisVin19}, there exists a complex constant $\alpha$ and a homomorphism $\varphi\colon \mathbb{F}\to \mathbb{C}$ such that 
\[
a(x)= \alpha \cdot \varphi(x) 
\qquad 
\left(x\in \mathbb{F}\right)
\]
and hence 
\[
f(x)= \alpha^{2n}\varphi(x)^{2}
\]
is fulfilled for all $x\in \mathbb{F}$.

The converse is an easy computation. 
\end{proof}

\begin{thm}
Let $\mathbb{F}\subset \mathbb{C}$ be a field, $f\colon \mathbb{F}\to \mathbb{C}$ be a quadratic function, while $a_{1}, a_{2}\colon \mathbb{F}\to \mathbb{C}$ be additive functions. 
Then equation 
\[
f(x^{2})= a_{1}(x)^{2}a_{2}(x)^{2}
\]
holds for all $x\in \mathbb{F}$ if and only if there exist homomorphisms $\varphi_{1}, \varphi_{2}\colon \mathbb{F}\to \mathbb{C}$ such that 
\[
f(x)= f(1)\cdot \varphi_{1}(x)\varphi_{2}(x) 
\qquad 
\text{and}
\qquad 
a_{i}(x)= a_{i}(1)\varphi_{i}(x) 
\qquad 
\left(x\in \mathbb{F}, i=1, 2\right). 
\]
\end{thm}

\begin{proof}
Let $\mathbb{F}\subset \mathbb{C}$ be a field, $f\colon \mathbb{F}\to \mathbb{C}$ be a quadratic function, while $a_{1}, a_{2}\colon \mathbb{F}\to \mathbb{C}$ be additive function. 
Suppose further that for all $x\in \mathbb{F}$, equation 
\[
f(x^{2})= a_{1}(x)^{2}a_{2}(x)^{2}
\]
is fulfilled. Since $f$ is quadratic, there exists a uniquely determined symmetric and bi-additive mapping $F$ such that 
\[
F(x, x)= f(x) 
\qquad 
\left(x\in \mathbb{F}\right). 
\]
Define the mapping $\Phi$ on $\mathbb{F}^{4}$ by 
\begin{multline*}
\Phi(x_{1}, x_{2}, x_{3}, x_{4})= 
\dfrac{F\left(x_{1}\,x_{4} , x_{2}\,x_{3}\right)+F\left(x_{1}\,x_{3} , 
 x_{2}\,x_{4}\right)+F\left(x_{1}\,x_{2} , x_{3}\,x_{4}\right)}{3}
 \\
 -{{a_{1}(x_{1})\,a_{1}(x_{2})\,a_{2}(x_{3})\,a_{2}(x_{4})+a_{1}(
 x_{1})\,a_{2}(x_{2})\,a_{1}(x_{3})\,a_{2}(x_{4})+a_{2}(x_{1})\,a_{1}
 (x_{2})\,a_{1}(x_{3})\,a_{2}(x_{4})}\over{6}}
 \\
 -{a_{1}(x_{1})\,a_{2}(x_{2})\,a_{2
 }(x_{3})\,a_{1}(x_{4})+a_{2}(x_{1})\,a_{1}(x_{2})\,a_{2}(x_{3})\,a_{
 1}(x_{4})+a_{2}(x_{1})\,a_{2}(x_{2})\,a_{1}(x_{3})\,a_{1}(x_{4})
 \over{6}}
 \\
 \left(x_{1}, x_{2}, x_{3}, x_{4}\in \mathbb{F}\right). 
\end{multline*}
Since $F$ is symmetric and bi-additive and $a_{1}$ and $a_{2}$ are additive, the function $\Phi$ is a symmetric $4$-additive mapping. Moreover its trace is 
\[
F\left(x^2 , x^2\right)-a_{1}(x)^2\,a_{2}(x)^2=0
\qquad 
\left(x\in \mathbb{F}\right). 
\]
Thus $\Phi$ is identically zero on $\mathbb{F}^{4}$. 
From this we especially get that 
\[
2F\left(x^2 , 1\right)-\,a_{1}(1)^2\,a_{2}(1)\,a_{2}(x^2)-\,a_{
 1}(1)\,a_{2}(1)^2\,a_{1}(x^2)= 0
\qquad 
\left(x\in \mathbb{F}\right)
\]
and also 
\[
2\,F\left(x^2 , 1\right)+4\,F\left(x , x\right)-a_{1}(1)^2\,a_{2}(x
 )^2-4\,a_{1}(1)\,a_{2}(1)\,a_{1}(x)\,a_{2}(x)-a_{2}(1)^2\,a_{1}(x)^2=0
 \qquad 
 \left(x\in \mathbb{F}\right). 
\]
These identities together imply that 
\begin{multline*}
4\,F\left(x , x\right)+a_{1}(1)^2\,a_{2}(1)\,a_{2}(x^2)+a_{1}(1)\,a_{2}(1)^2\,a_{1}(x^2)-a_{1}(1)^2\,a_{2}(x)^2
\\
-4\,a_{1}(1)\,a_{2}(1)\,
 a_{1}(x)\,a_{2}(x)-a_{2}(1)^2\,a_{1}(x)^2=0
\\ 
\left(x\in \mathbb{F}\right). 
\end{multline*}
In other words, we have 
\[
4F(x, x)= -a_{1}(1)^2\,a_{2}(1)\,a_{2}(x^2)-a_{1}(1)\,a_{2}(1)^2\,a_{1}(x^2)+
 a_{1}(1)^2\,a_{2}(x)^2+4\,a_{1}(1)\,a_{2}(1)\,a_{1}(x)\,a_{2}(x)+a_{
 2}(1)^2\,a_{1}(x)^2
\]
for all $x\in \mathbb{F}$. Since $F$ is a symmetric and bi-additive mapping, we have 
\begin{multline*}
4F(x, y)= 
-a_{1}(1)^2\,a_{2}(1)\,a_{2}(x\,y)-a_{1}(1)\,a_{2}(1)^2\,a_{1}(x\,y
 )
 \\
 +2\,\left(a_{1}(1)\,a_{2}(1)\,a_{1}(x)\,a_{2}(y)+a_{1}(1)\,a_{2}(1)
 \,a_{2}(x)\,a_{1}(y)\right)+a_{1}(1)^2\,a_{2}(x)\,a_{2}(y)+a_{2}(1)^
 2\,a_{1}(x)\,a_{1}(y)
\end{multline*}
for all $x, y\in \mathbb{F}$. 
Combining this with our functional equation, we deduce 
\begin{multline*}
    -a_{1}(1)^2\,a_{2}(1)\,a_{2}(x^4)-a_{1}(1)\,a_{2}(1)^2\,a_{1}(x^4)+
 a_{1}(1)^2\,a_{2}(x^2)^2
 \\
 +4\,a_{1}(1)\,a_{2}(1)\,a_{1}(x^2)\,a_{2}(x^
 2)+a_{2}(1)^2\,a_{1}(x^2)^2-4\,a_{1}(x)^2\,a_{2}(x)^2=0
 \\ 
(x\in \mathbb{F}).
\end{multline*}
Observe that if $a_{1}(1)\cdot a_{2}(1)=0$, then $F$ and thus $f$ is identically zero. So we may assume $a_{1}(1)\cdot a_{2}(1)\neq 0$. 
Without the loss of generality we can then suppose that $a_{1}(1)=a_{1}(1)=1$, otherwise we consider the mappings 
\[
\tilde{f}(x)= \frac{f(x)}{a_{1}(1)\cdot a_{2}(1)}
\qquad 
\text{and}
\qquad 
\tilde{a_{i}}(x)= \frac{a_{i}(x)}{a_{i}(1)} 
\qquad 
\left(x\in \mathbb{F}, i=1, 2\right). 
\]
Then this equation reduces to 
\[
a_{2}(x^4)+a_{1}(x^4)=a_{2}(x^2)^2+a_{1}
 (x^2)^2+4\,a_{1}(x^2)\,a_{2}(x^2)-4\,a_{1}(x)^2\,a_{2}(x)^2
 \qquad 
 \left(x\in \mathbb{F}\right). 
\]
After symmetrization, we get that this is equivalent to 
\begin{multline*}
    -a_{2}(x_{1}\,x_{2}\,x_{3}\,x_{4})-a_{1}(x_{1}\,x_{2}\,x_{3}\,x_{4}
 )
 \\
 +\dfrac{1}{3}\left[{a_{2}(x_{1}\,x_{2})\,a_{2}(x_{3}\,x_{4})+a_{2}(x_{1}\,x_{3})\,a
 _{2}(x_{2}\,x_{4})+a_{2}(x_{2}\,x_{3})\,a_{2}(x_{1}\,x_{4})}\right]
 \\
 +\frac{2}{3} \left[a_{1}(x_{1}\,x_{2})\,a_{2}(x_{3}\,x_{4})+a_{2}(x_{1}\,
 x_{2})\,a_{1}(x_{3}\,x_{4})+a_{1}(x_{1}\,x_{3})\,a_{2}(x_{2}\,x_{4})
 \right. 
 \\
 \left. 
 +a_{2}(x_{1}\,x_{3})\,a_{1}(x_{2}\,x_{4})+a_{1}(x_{2}\,x_{3})\,a_{2}
 (x_{1}\,x_{4})+a_{2}(x_{2}\,x_{3})\,a_{1}(x_{1}\,x_{4})
 \right]
 \\
 +\dfrac{1}{3}\left[{a_{1}(x_{1}\,x_{2})\,a_{1}(x_{3}\,x_{4})+a_{1}(x_{1}\,
 x_{3})\,a_{1}(x_{2}\,x_{4})+a_{1}(x_{2}\,x_{3})\,a_{1}(x_{1}\,x_{4})
 }\right]
 \\
 -\dfrac{2}{3}\,\left[a_{1}(x_{1})\,a_{1}(x_{2})\,a_{2}(x_{3})\,a_{2
 }(x_{4})+a_{1}(x_{1})\,a_{2}(x_{2})\,a_{1}(x_{3})\,a_{2}(x_{4})+a_{2
 }(x_{1})\,a_{1}(x_{2})\,a_{1}(x_{3})\,a_{2}(x_{4})
 \right. 
 \\
 \left. 
 +a_{1}(x_{1})\,a_{
 2}(x_{2})\,a_{2}(x_{3})\,a_{1}(x_{4})+a_{2}(x_{1})\,a_{1}(x_{2})\,a
 _{2}(x_{3})\,a_{1}(x_{4})+a_{2}(x_{1})\,a_{2}(x_{2})\,a_{1}(x_{3})\,
 a_{1}(x_{4})\right]=0
 \\
 \left(x_{1}, x_{2}, x_{3}, x_{4}\in \mathbb{F}\right). 
\end{multline*}
From this we deduce that for all $x, y, z\in \mathbb{F}$ we have 
\begin{multline*}
    3\,a_{2}(x\,y\,z)+3\,a_{1}(x\,y\,z)
    \\
    =
    a_{2}(x)\,a_{2}(y\,z)+2\,a_{1}
 (x)\,a_{2}(y\,z)+2\,a_{2}(x)\,a_{1}(y\,z)+a_{1}(x)\,a_{1}(y\,z)+a_{2
 }(y)\,a_{2}(x\,z)
 \\
 +2\,a_{1}(y)\,a_{2}(x\,z)+2\,a_{2}(y)\,a_{1}(x\,z)+
 a_{1}(y)\,a_{1}(x\,z)+a_{2}(x\,y)\,a_{2}(z)+2\,a_{1}(x\,y)\,a_{2}(z)
 \\
 -2\,a_{1}(x)\,a_{2}(y)\,a_{2}(z)-2\,a_{2}(x)\,a_{1}(y)\,a_{2}(z)-2\,
 a_{1}(x)\,a_{1}(y)\,a_{2}(z)+2\,a_{2}(x\,y)\,a_{1}(z)+a_{1}(x\,y)\,a
 _{1}(z)
 \\
 -2\,a_{2}(x)\,a_{2}(y)\,a_{1}(z)-2\,a_{1}(x)\,a_{2}(y)\,a_{1}
 (z)-2\,a_{2}(x)\,a_{1}(y)\,a_{1}(z)
\end{multline*}
For all $z\in \mathbb{F}$, define the additive function by 
\[
A_{z}(x)= 3a_{1}(xz)-a_{1}(x)\left[a_{1}(z)+2a_{2}(x)\right]
+3a_{2}(xz)-a_{2}(x)\left[a_{2}(z)+2a_{1}(x)\right]
\qquad 
\left(x\in \mathbb{F}\right). 
\]
With this notation the above identity turns to 
\[
A_{z}(xy)= a_{1}(x)g_{z}(y)+a_{2}(x)h_{z}(y)+k_{z}(x)a_{1}(y)+l_{z}(x)a_{2}(y) 
\qquad 
\left(x, y, z\in \mathbb{F}\right), 
\]
with appropriately defined functions $g_{z}, h_{z}, k_{z}, l_{z}$, yielding that the mapping $A_{z}$ is a (normal) exponential polynomial of degree at most $4$ on the multiplicative group $\mathbb{F}^{\times}$. Further, 
\begin{enumerate}[(A)]
    \item either the system $a_{1}, a_{2}, k_{z}, l_{z}$ is linearly dependent
    \item or the system $a_{1}, a_{2}, k_{z}, l_{z}$ is linearly independent. 
\end{enumerate}
Alternative $(A)$ holds only if the functions $a_{1}$ and $a_{2}$ are the same (note that we assumed that $a_{1}(1)=a_{2}(1)=1$), but then Theorem \ref{thm7} applies and we deduce that there exists a homomorphism $\varphi\colon \mathbb{F}\to \mathbb{C}$ such that 
\[
f(x)= f(1)\varphi(x)^{2} 
\qquad 
a_{i}(x)= a_{i}(1)\varphi(x) 
\qquad 
\left(x\in \mathbb{F}\right). 
\]
If alternative $(B)$ holds, then we get that not only the mapping $A_{z}$, but also the functions $a_{1}, a_{2}, k_{z}, l_{z}$ are (normal) exponential polynomials on the multiplicative group $\mathbb{F}^{\times}$ of degree at most four. Especially, $a_{1}$ and $a_{2}$ are linearly independent (normal) exponential polynomials of degree at most four. Checking all the possibilities, we finally get that this can happen only if $a_{1}$ and $a_{2}$ are exponentials on the multiplicative group $\mathbb{F}^{\times}$. Since these functions were assumed also to be additive on $\mathbb{F}$, we get that they are homomorphisms on $\mathbb{F}$. 

Suming up, there exists homomorphisms $\varphi_{1}, \varphi_{2}\colon \mathbb{F}\to \mathbb{C}$ such that 
\[
f(x)= f(1)\varphi_{1}(x)\varphi_{2}(x)
\qquad 
a_{i}(x)= a_{i}(1)\varphi_{i}(x) 
\qquad 
\left(x\in \mathbb{F}\right)
\]
where the above constants $f(1), a_{1}(1)$ and $a_{2}(1)$ also fulfill $f(1)= a_{1}(1)a_{2}(1)$. 
\end{proof}

\begin{ackn}
The research of E.~Gselmann has been supported by project no.~K134191 that has been
implemented with the support provided by the National Research, Development and Innovation Fund of Hungary, financed under the K{\_}20 funding scheme.
\end{ackn}

\newpage
\noindent
\textbf{Eszter Gselmann} \\
Department of Analysis\\
University of Debrecen\\
P.O. Box 400\\
H-4002 Debrecen\\
Hungary\\
e-mail: gselmann@science.unideb.hu\\
ORCID: 0000-0002-1708-2570
\vspace{1cm}

\noindent
\textbf{Mehak Iqbal}\\
Doctoral School of Mathematical and Computational Sciences\\
University of Debrecen\\
P.O. Box 400\\
H-4002 Debrecen\\
Hungary\\
e-mail: iqbal.mehak@science.unideb.hu\\

\end{document}